\documentclass[11pt]{amsart}

\usepackage[english]{babel}
\usepackage[utf8]{inputenc}
\usepackage{amsmath}
\usepackage{graphicx}
\usepackage{amssymb}
\usepackage{amsthm}
\usepackage{tikz-cd}
\usepackage{mathrsfs}
\usepackage[colorinlistoftodos]{todonotes}
\usepackage{enumitem}
\usepackage{yfonts}

\newtheorem{thm}{Theorem}[section]
\newtheorem{lem}[thm]{Lemma}
\newtheorem{prop}[thm]{Proposition}
\newtheorem{defn}[thm]{Definition}

\newtheorem{cor}[thm]{Corollary}

\newtheorem{ques}[thm]{Question}

\begin{document}

\title{Higher homotopy groups of Cuntz classes}

\author[AT]{Andrew S. Toms}

\address{Department of Mathematics, Purdue University, 150 N University St, West Lafayette IN 47907, USA \ {\tt atoms@purdue.edu}}

\maketitle

\begin{abstract}

Let $A$ be a unital simple separable exact \mbox{C$^*$-algebra} which is approximately divisible and of real rank zero.  We prove that the set of positive elements in $A$ with a fixed non-compact Cuntz class has vanishing homotopy groups.  Combined with work of S. Zhang for the case of compact elements, this gives a complete calculation of the homotopy groups of Cuntz classes for these algebras. Examples covered include approximately finite-dimensional (AF) algebras and irrational noncommutative tori.  
\end{abstract}

\section{Introduction}\label{intro}

Noncommutative tori are fundamental examples in the theory of operator algebras and noncommutative geometry.   When simple they are amenable to classification via $K$-theory in the sense of Elliott, and they are the primordial instance of a noncommutative differentiable manifold.  The classification of irrational rotation algebras via $K$-theory due to Elliott and Evans was an inflection point in the momentum of Elliott's attendant classification program for nuclear separable simple C$^*$-algebras (\cite{E}, \cite{EE}), and was preceded by the discovery of several important structural properties.  Blackadar, Kumjian and R\o rdam established real rank zero for these algebras through their introduction and application of approximate divisiblity, while Rieffel studied their nonstable $K$-theory (\cite{BKR}, \cite{Ri}).  In particular, Rieffel proved that the higher homotopy groups of the unitary groups of a simple $A_\theta$ were stable and alternated between $K_0(A_\theta)$ and $K_1(A_\theta)$, a manifestation of Bott periodicity.  Building on this work, Zhang later studied the higher homotopy groups of the set $\mathcal{P}(A)$ of nontrivial projections in a simple unital C$^*$-algebra $A$ with cancellation of projections and real rank zero, properties enjoyed by the simple $A_\theta$ (\cite{Z1}).  For these algebras it was shown that these homotopy groups again alternate between $K_0(A)$ and $K_1(A)$.  


In the sequel we return to the study of simple noncommutative tori (in fact, to a much larger class which contains them naturally) through the lens of an invariant whose importance has come to the fore since the early 2000s, namely, the Cuntz semigroup (\cite{Cuntz1}).  
This invariant can distinguish simple nuclear separable and non-isomorphic C$^*$-algeberas which nevertheless agree on K-theory and traces, a discovery that catalyzed the study of regularity properties in Elliott's classification program ({\cite{ET}, \cite{To1}, \cite{TW3}).  
Roughly speaking, the Cuntz semigroup generalizes the Murray-von Neumann semigroup for projections to the setting of positive elements.
The Cuntz class of a positive element is akin to the Murray-von Neumann equivalence class of a projection, and under the assumption of stable finiteness the two classes agree for projections.  For well-behaved C$^*$-algebras such as those studied by Zhang and herein, the Murray-von Neumann class and the path component of a projection in $\mathcal{P}(A)$ are synonymous, and so Zhang's result can be interpreted as a calculation of the higher homotopy groups of the Cuntz classes of projections.  It is therefore natural to consider what these groups might be for the Cuntz class of a general positive element.  The answer, it turns out, is quite satisfying.



For a C$^*$-algebra $A$ and positive $a \in A$ we let $\langle a \rangle$ denote the Cuntz class of $a$ (see Section \ref{prelim}).
\begin{defn}\label{constantclass}
Let $A$ be a C$^*$-algebra and $X$ a compact Hausdorff space.  For each positive element $a$ of $A$ set 
\[
D_{a,X} = \{ f \in C(X,A) \ | \ f \geq 0 \ \mathrm{and} \ \langle f(x) \rangle = \langle a \rangle, \ \forall x \in X \}.
\] 
\end{defn}
\noindent
We will simply write $D_a$ in place of $D_{a,X}$ when the space $X$ is a singleton.

\begin{thm}\label{main}
Let $A$ be a unital simple separable exact \mbox{C$^*$-algebra} which is approximately divisible and of real rank zero.  Let $a \in A$ be a positive element such that $\langle a \rangle$ is not compact in $\mathrm{Cu}(A)$ (if $A$ is stably finite) or nonzero (if $A$ is purely infinite).  It follows that $D_{a,X}$ is path connected for every compact metric space $X$.  
\end{thm}

\noindent
Setting $X = S^n$ in Theorem \ref{main} we obtain the following Corollary:

\begin{cor}\label{homotopyzero}
The homotopy groups $\pi_k(D_a)$ vanish for all $k \geq 1$.
\end{cor}

We expect that the conditions of approximate divisibiltiy and real rank zero in Theorem \ref{main} can be replaced with the condition of $\mathcal{Z}$-stability, so that Corollary \ref{homotopyzero} will apply to all algebras which are classifiable in the sense of Elliott.
While $K$-theory is essential to the proof of our main result, it must be employed delicately.  Cuntz classes lack the robustness of \mbox{$K_0$-classes} under norm perturbations.  A line segment joining two projections which are close in norm is easily deformed to a homotopy of projections which is necessarily of constant Murray-von Neumann class, while nothing of the sort holds among positive elements and their Cuntz classes.  Indeed, for a unital C$^*$-algebra $A$, the set of positive elements whose Cuntz class is equal to that of the the unit $1_A$ is dense among all positive elements.

Combining Theorem \ref{main} with Zhang's results from \cite{Z1} we have the following summary:
\vspace*{20mm}
\pagebreak

\begin{cor}\label{cuntzhom}
Let $A$ be a unital simple separable exact \mbox{C$^*$-algebra} which is approximately divisible and of real rank zero. Let $a \in A$ be positive.  It follows that $D_a$ is path connected.  If moreoever, $\langle a \rangle$ is compact in the Cuntz semigroup $\mathrm{Cu}(A)$, then for $k \in \mathbb{N}$ we have:
\[
\pi_k(D_a) = \left\{ \begin{array}{ll} K_0(A), & k \ \mathrm{even} \\ K_1(A), & k \ \mathrm{odd} \end{array} \right.
\]
If, alternatively, $\langle a \rangle$ is non-compact, then $\pi_k(D_a) = 0$ for each $k \geq 1$.
\end{cor}


Another way to phrase Corollary \ref{homotopyzero} is to say that $D_a$ is weakly contractible when $\langle a \rangle$ is not compact.  It would of course be desirable to upgrade this to genuine contractibility, which would follow from a positive answer to the following question:

\begin{ques}\label{cw}
Does a Cuntz class $D_a$ as in Theorem \ref{main} have the homotopy type of a CW-complex?
\end{ques}

The paper is organized as follows:  Section \ref{prelim} reviews essential facts about the Cuntz semigroup and approximate divisibility;  Section \ref{redux} reduces the proof of Theorem \ref{main} to the case where $a$ is a sum of scaled orthogonal projections;  Section \ref{mainproof} completes the proof of Theorem \ref{main}.

\section{Preliminaries}\label{prelim}

We survey briefly those aspects of the Cuntz semigroup essential to the sequel, and refer the reader to \cite{APT} for a detailed exposition.  Let $A$ be a $C^*$-algebra and let $A_+$ denote the subset of positive elements.  If $a,b \in A_+$, we write $a \precsim b$ if there is a sequence $(v_n)$ in $A$ such that $v_n b v_n^* \to a$ with convergence in norm.  If $a \precsim b$ and $b \precsim a$ then we write $a \sim b$ and say that $a$ and $b$ are Cuntz equivalent.  Set $\mathrm{Cu}(A) =( A \otimes \mathcal{K})_+ / \sim$, and let $\langle a \rangle$ denote the equivalence class of $a$ in $\mathrm{Cu}(A)$.  Now 
\[
\langle a \rangle+\langle b \rangle := \left\langle \left[ \begin{array}{cc} a & 0 \\ 0 & b \end{array} \right] \right\rangle
\]
defines addition on $\mathrm{Cu}(A)$ upon identifying $\mathcal{K}$ with $M_2(\mathcal{K})$, and 
\[
\langle a \rangle \leq \langle b \rangle \Leftrightarrow a \precsim b
\]
defines a partial order.
These make $\mathrm{Cu}(A)$ into an ordered semigroup, the Cuntz semigroup.

Suprema for increasing sequences always exist in $\mathrm{Cu}(A)$ (\cite{CEI}).  We say that $\langle a \rangle$ is compactly contained in $ \langle b \rangle$, written $\langle a \rangle \ll \langle b \rangle$, if whenever $(\langle b_i \rangle)$ is increasing and $\sup \langle b_i \rangle \geq \langle b \rangle$ we have $\langle a \rangle \leq \langle b_{i_0}\rangle$ for some $i_0 \in \mathbb{N}$.  An element $\langle a \rangle$ of $\mathrm{Cu}(A)$ is called compact if $\langle a \rangle \ll \langle a \rangle$.  If $A$ is simple or of stable rank one, then $\langle a \rangle$ is compact if and only if $a \sim p$ for some projection $p$ (\cite{APT}, \cite{BPT}).  Recall that $A$ has real rank zero if every self-adjoint element in $A$ can be approximated in norm by self-adjoint elements with finite spectrum.  Perera shows that consequently, for such algebras, each $\langle a \rangle \in \mathrm{Cu}(A)$ is seen to be the supremum of a sequence of compact elements by exploiting (i) and (iii) below (\cite{P}).

Let $\delta>0$, and define a continuous map $g_\epsilon: \mathbb{R} \to [0,\infty)$ by 
\[
g_\epsilon(t) = \max \{ 0,t-\epsilon \}.
\]
Set $(a-\epsilon)_+ = g_\epsilon(a)$.  We record some useful facts related to this construction (see \cite{APT}):
\begin{enumerate}
\item[(i)] $a \precsim b$ if and only if $(a-\epsilon)_+ \precsim b$, for every $\epsilon >0$;
\item[(ii)] $\langle (a-\epsilon)_+ \rangle \ll \langle a \rangle$ for every $\epsilon >0$;
\item[(iii)] If $\|a-b\| < \epsilon$, then $(a-\epsilon)_+ \precsim b$ for every $\epsilon >0$;
\item[(iv)] If $a \precsim b$ then for every $\epsilon>0 $ there is $\delta >0$ such that \\ $(a-\epsilon)_+ \precsim (b-\delta)_+$.
\end{enumerate}
The next Lemma is Proposition 2.7 (i) of \cite{KR}.
\begin{lem}\label{herrep}
Let $A$ be a C$^*$-algebra with $a,b \in A_+$.  If $b \in \overline{aAa}$, then $b \precsim a$.
\end{lem}


Let $A$ be unital and exact, and let $\mathrm{T}(A)$ denote the space of tracial states on $A$.  For $\tau \in \mathrm{T}(A)$ and $a \in (A \otimes \mathcal{K})_+$ we recall that
\[
d_\tau(a) = \lim_{n \to \infty} \tau(a^{1/n}).
\]
defines a lower semicontinuous dimension function on $A$.  We say that $A$ has strict comparison of positive elements if $a \precsim b$ whenever
\[
d_\tau(a) < d_\tau(b), \ \forall \tau \in \mathrm{T}(A).
\]
(Strictly speaking we should consider all possible $\tau$ belonging to the set of 2-quasitraces on $A$, denoted by $QT(A)$, but we are assuming that $A$ is exact so that $QT(A)=T(A)$ by a result of Haagerup---see \cite{H}.)

A unital separable C$^*$-algebra $A$ is said to be {\it approximately divisible} if there exists, for each $N \in \mathbb{N}$, an approximately central sequence of unital $*$-homomorphisms $\phi_N: M_N(\mathbb{C}) \oplus M_{N+1}(\mathbb{C}) \to A$ (\cite{BKR}).  It is known that unital simple infinite-dimensional AF algebras are approximately divisible, as are simple noncommutative tori, and that the latter have real rank zero (\cite{BKR}, Theorem 1.5).  Approximately divisible C$^*$-algebras are $\mathcal{Z}$-stable, and simple $\mathcal{Z}$-stable C$^*$-algebras have strict comparison of positive elements (\cite{Ro1}, \cite{TW2}).  It is noted in \cite{BKR} that if $A$ is approximately divisible then so is $A \otimes B$ for separable $B$, whence 
\[
C(X,A) \cong C(X) \otimes A
\]
is approximately divisible for every compact metric space $X$.

\section{Reduction to sums of projections}\label{redux}

The next result follows from Lemma 3.2 of \cite{To2} and Proposition 3.9 of \cite{A}.  The latter result is a Corollary of Proposition 3.2 of \cite{Betc}.

\begin{lem}\label{split}
Let $A$ be a unital simple separable exact $\mathcal{Z}$-stable C$^*$-algebra which is moreover (residually) stably finite, and let $X$ be a compact metric space.  Suppose that $f \in C(X,A)$ is positive, and that $q \in C(X,A) \otimes \mathcal{K}$ is a projection satisfying $d_\tau(q) < d_\tau(f)$ for every $\tau \in \mathrm{T}(C(X,A)$.  It follows that there are a projection $p \sim q$ and a positive element $b$ in $\overline{f(C(X,A))f}$ such that $bp = pb =0$ and
\[
d_\tau(a) = d_\tau(b) + d_\tau(p), \ \forall \tau \in \mathrm{T}(C(X,A)).
\]
\end{lem}

\begin{proof}
We observe first that the assumption of simplicity in Lemma 3.2 of \cite{To2} can be dropped. By Proposition 3.9 of \cite{A} we have that $C(X,A)$ enjoys strict comparison of positive elements, whence $q \precsim f$.  One now follows the proof of Lemma 3.1 of \cite{To2} verbatim with $C(X,A)$ in place of $A$.
\end{proof}

The following Proposition generalizes Proposition 3.3 of \cite{To2}.

\begin{prop}\label{projsum}
Let $A$ be a unital simple separable exact C$^*$-algebra which is $\mathcal{Z}$-stable and of real rank zero.  Suppose further that $A$ is stably finite, and let $X$ be a compact metric space.  Let $0 \neq a \in A $ be positive, and let $f \in C(X,A)$ be positive and satisfy $\langle f(x) \rangle = \langle a \rangle$ for every $x \in X$.  It follows that there is a sequence $p_1,p_2,\ldots$ of mutually orthogonal projections in $\overline{f((C(X,A))f}$ such that with
\[
b = \bigoplus_{i=1}^\infty \frac{1}{2^i} p_i,
\]
we have a homotopy $h:[0,1] \to (\overline{f(C(X,A))f})_+$ satisfying $h(0) = f$, $h(1) = b$, and $\langle h(t)(x) \rangle = \langle a \rangle \in \mathrm{Cu}(A)$ for each $t \in [0,1]$ and $x \in X$.  If $\langle a \rangle$ is non-compact then we may moreover assume that 
\[
0 < \tau(p_i) < d_\tau(f), \ \forall \tau \in \mathrm{T}(C(X,A)).
\]

\end{prop}

\begin{proof}

If $\langle a \rangle = \langle f(x) \rangle \in \mathrm{Cu}(A) $ is compact, then for each $x \in X$ there is $\epsilon_x > 0$ such that 
\[
\langle (f(x)-\epsilon_x)_+ \rangle  \sim \langle f(x) \rangle.
\]
Since $A$ is simple, we conclude that the interval $(0,\epsilon_x)$ does not meet the spectrum $\sigma(f(x))$ of $f(x)$ (\cite{BPT}).  It then follows from the continuity of \mbox{$f: X \mapsto A$} that there is an open set $U_x$ containing $x$ such that the interval $(0,\epsilon_x/2)$ does not meet $\sigma(f(y))$ for each $y \in U_x$, whence
\[
\langle (f(y) - \epsilon_x/2)_+ \rangle \sim \langle f(y) \rangle, \ \forall y \in U_x.
\]
By the compactness of $X$ we can find $x_1,\ldots,x_n \in X$ such that 
\[
\cup_{i=1}^n U_{x_i} = X.
\]
Setting $\epsilon = \min_i \{\epsilon_i\}$ we have that
\[
\langle (f(x) - \epsilon)_+ \rangle \sim \langle f(x) \rangle, \ \forall x \in X.
\]
By construction, we have that the spectrum of each $f(x)$ does not intersect $(0,\epsilon)$, i.e., that zero is an isolated point of the spectrum of $f$.  Lemma 2.2 of \cite{To2} then yields the desired homotopy between $f(x)$ and a projection $p \in C(X,A)$. 

Assume, then, that $\langle a \rangle$ is non-compact.  Real rank zero implies the existence of a sequence $q_1,q_2,\ldots$ of projections in $A \otimes \mathcal{K}$ such that $q_i \neq q_{i+1}$, $\langle q_i \rangle \ll \langle q_{i+1} \rangle$, and 
\begin{equation}\label{asup}
\sup_i \langle q_i \rangle = \langle a \rangle
\end{equation}
\noindent
(see \cite{P}).  Since $A$ is $\mathcal{Z}$-stable, the $\mathrm{K}_0$-group of $A$ is weakly unperforated (\cite{BKR}).  It follows that 
\[
[ q_{i} ] = [ q_{i-1} ]  + [ r_i ]
\]
for each $i >1$ and some projection $0 \neq r_i \in A \otimes \mathcal{K}$, where $[  \bullet  ]$ denotes the $\mathrm{K}_0$-class.  Set $r_1 = q_1$. By Theorem 2.2 of \cite{BPT}, the positive cone of $\mathrm{K}_0(A)$ embeds faithfully as the subsemigroup of compact elements in $\mathrm{Cu}(A)$, so that
\begin{equation}\label{qstrict}
\langle q_{i} \rangle = \langle q_{i-1} \rangle  + \langle r_i \rangle.
\end{equation}

From here on we abuse notation slightly and use $r_i$ to refer the the constant map $r_i:X \to A \otimes \mathcal{K}$ which is everywhere equal to $r_i \in A \otimes \mathcal{K}$.  We do likewise for $q_i$.  Since $r_i$ is everywhere nonzero on $X$ it follows from (\ref{asup}) and (\ref{qstrict}) that
\begin{equation}\label{qistrict}
d_\tau(q_i(x)) < d_\tau((f(x))
\end{equation}
and
\begin{equation}\label{ristrict}
d_\tau(q_i(x)) = d_\tau(q_{i-1}(x)) + d_\tau(r_i(x))
\end{equation}
for every $x \in X$.  It is well known that every extreme tracial state on 
\[
C(X,A) \cong C(X) \otimes A
\]
is of the form $\mathrm{ev}_x \otimes \gamma$, where $\mathrm{ev}_x$ denotes evaluation at $x \in X$ and $\gamma \in \mathrm{T}(A)$ is extreme.  We may thus conclude from (\ref{qistrict}) and (\ref{ristrict}) that
\begin{equation}\label{qifstrict}
d_\tau(q_i) < d_\tau(f)
\end{equation}
and 
\begin{equation}\label{sumstrict}
d_\tau(q_i) = d_\tau(q_{i-1}) + d_\tau(r_i)
\end{equation}
for every $\tau \in \mathrm{T}(C(X,A))$.
 
With (\ref{qifstrict}) and (\ref{sumstrict}) in hand we
apply Lemma \ref{split} with $q = r_1$ and $f$ as in the present Proposition. This yields a projection $p_1 \sim r_1$ and positive element $b_1$ in $\overline{f(C(X,A))f}$ such that $p_1 b_1 = 0$ and 
\[
d_\tau(f) = d_\tau(p_1) + d_\tau(b_1), \ \forall \tau \in \mathrm{T}(C(X,A)).
\]
Assume, inductively, that we have found in $\overline{f(C(X,A))f}$ mutually orthogonal projections $p_i \sim r_i$, $1 \leq i \leq n$,  and a positive element $b_n$ orthogonal to each $p_i$ such that
\begin{equation}\label{induct}
d_\tau(f) = d_\tau \left( \oplus_{i=1}^n p_i \right) + d_\tau(b_n) = \sum_{i=1}^n  d_\tau(p_i) + d_\tau(b_n).
\end{equation}
Appealing again to (\ref{qifstrict}) and (\ref{sumstrict}) we apply Lemma \ref{split} with $f=b_n$ and $q=r_{n+1}$.  We label the resulting projection and positive element as $p_{n+1}$ and $b_{n+1}$, respectively, and this yields (\ref{induct}) above with $n$ replaced by \mbox{$n+1$}.  Continuing inductively in this manner produces a sequence of mutually orthogonal projections $p_1,p_2,\ldots$ in $\overline{f(C(X,A))f}$.  Note that since $p_i \sim q_i$ we conclude from (\ref{qifstrict}) and (\ref{sumstrict}) that the last conclusion of the Proposition holds.

Set 
\[
b = \bigoplus_{i=1}^\infty \frac{1}{2^i} p_i.
\]
Since $b \in \overline{f(C(X,A))f}$ we have $b \precsim f$, whence 
\[
b(x) \precsim f(x) \sim a, \ \forall x \in X.
\]
It follows that
\[
d_\tau(b(x)) \leq d_\tau(f(x)) = d_\tau(a), \ \forall x \in X, \ \forall \tau \in \mathrm{T}(A).
\] 
Note that zero is an accumulation point of the spectrum of $b(x)$ for each $x \in X$, so that $b(x)$ is not compact in $\mathrm{Cu}(A)$.  

Now
\[
q_n \sim \oplus_{i=1}^n p_i \leq b 
\]
by construction so that
\[
d_\tau(q_n(x)) \leq d_\tau(b(x)), \ \forall \tau \in \mathrm{T}(A), \ \forall x \in X, \ \forall n \in \mathbb{N}.
\]
Since $\langle a \rangle = \sup_n \langle q_n(x) \rangle$ in $\mathrm{Cu}(A)$ we then have
\[
d_\tau(a) = \sup_n d_\tau(q_n(x)) \leq d_\tau(b(x)) \leq d_\tau(a), \ \forall \tau \in \mathrm{T}(A), \ \forall x \in X,
\]
and so 
\[ 
d_\tau(b(x)) = d_\tau(a), \  \forall \tau \in \mathrm{T}(A), \ \forall x \in X.
\]
$A$ is $\mathcal{Z}$-stable, and so has strict comparison of positive elements (\cite{Ro1}, \cite{TW2}).  Since neither $a$ nor $b$ is compact it follows that $\langle a \rangle= \langle b(x) \rangle$ in $\mathrm{Cu}(A)$ for each $x \in X$ (\cite{BPT}, Theorem 2.2).

Set 
\[
h(t) = (1-t)f +tb, \ t \in [0,1].
\]
Then,
\[
\langle h(0)(x) \rangle = \langle f(x) \rangle = \langle a \rangle
\]
and
\[
\langle h(1)(x) \rangle =  \langle b(x) \rangle = \langle a \rangle
\]
for every $x \in X$.
For $t \in (0,1)$ we note that $h(t) \in \overline{f(C(X,A))f}$ since each of $f$ and $b$ is, whence
\[
\langle h(t)(x) \rangle \leq \langle f(x) \rangle = \langle a \rangle.  
\]
On the other hand
\[
a \sim (1-t)a \precsim (1-t)a + tb(x) = h(t)(x),
\]
so that $ \langle h(t)(x) \rangle \geq \langle a \rangle$ for each $x \in X$ also.  We conclude that
\[
\langle h(t)(x) \rangle = \langle a \rangle, \ \forall t \in [0,1], \ \forall x \in X,
\]
as required.
\end{proof}

We record the following Corollary of (the proof of) Proposition \ref{projsum} for use in the proof of Theorem \ref{main} below.

\begin{cor}
Let $A$ and $a$ be as in Proposition \ref{projsum}. Suppose that 
\[
f,g \in C(X,A) \otimes \mathcal{K}
\]
are positive and satisfy 
\[
\langle f(x) \rangle = \langle g(x) \rangle = \langle a \rangle, \ \forall x \in X.
\]
Let $(p_{i,f})_{i \in \mathbb{N}}$ and $(p_{i,g})_{i \in \mathbb{N}}$ be the sequences of projections provided by Proposition \ref{projsum} for $f$ and $g$, respectively.  It follows that $p_{i,f} \sim p_{i,g}$ for every $i \in \mathbb{N}$.
\end{cor}

\begin{proof}
The choice of the constant projection valued maps $r_i :X \to A \otimes \mathcal{K}$ in the proof of Proposition \ref{projsum} depends only on the Cuntz class $\langle a \rangle$ and by construction we have 
\[
p_{i,f} \sim r_i \sim p_{i,g}.
\]
\end{proof}

\section{The homotopy groups of Cuntz classes}\label{mainproof}

We review briefly here the preamble to the proof of Theorem 1.1 of \cite{To2}, adapted to our purpose.  This proof provides the framework for the proof of Theorem \ref{main}. 
Recall that a unital C$^*$-algebra $A$ is {\it strongly $\mathrm{K}_1$-surjective} if the canonical map 
\[
\mathcal{U}(B + \mathbb{C}(1_A)) \to \mathrm{K}_1(A)
\]
is surjective for every full hereditary subalgebra $B$ of $A$.
The following Lemma is due to R\o rdam (\cite{Ro1}, Lemma 6.3).

\begin{lem}\label{K1rep}
Let $A$ be a unital approximately divisible C$^*$-algebra.  Then $A$ is strongly $\mathrm{K}_1$-surjective.  In particular, for any full projection $p \in A$ and $g \in \mathrm{K}_1(A)$, there is a unitary $v \in pAp$ such that $g = [v+ (1_A-p)]_1$ in $\mathrm{K}_1(A)$.
\end{lem}

Proposition 3.10 of \cite{BKR} states that a unital approximately divisible C$^*$-algebra has cancellation, so Lemma 4.2 of \cite{To2} does not require simplicity:

\begin{lem}\label{projhom}
Let $A$ be a unital approximately divisible C$^*$-algebra and let $p, q$ be full projections in $A$ with $[p] =[q]$ in $\mathrm{K}_0(A)$.  It follows that there is a continuous path of unitaries $(u_t)_{t \in [0,1]}$ in $A$ such that $u_0 = 1_A$ and $u_1 p u_1^* = q$.  In particular, $p$ and $q$ are homotopic via projections in $A$.
\end{lem}



We can now prove our main result.

\begin{proof} 
(Theorem \ref{main})
Let $A$ be a simple unital separable exact C$^*$-algebra which is approximately divisible and of real rank zero.  Since $A$ is $\mathcal{Z}$-stable by Theorem 2.3 of \cite{TW2}, it is either stably finite (when it has a trace) or purely infinite (when it is traceless) (\cite{Ro1}).  

Consider first the case where $A$ is purely infinite, and let $a \in A$ be a nonzero positive element.  Let $f,g \in C(X,A)$ be positive and satisfy
\begin{equation}\label{cuntzequal}
\langle f(x) \rangle = \langle a \rangle = \langle g(x) \rangle, \ \forall x \in X.
\end{equation}
Since $A$ is purely infinite, the Cuntz classes of any two nonzero positive elements are the same.  It follows from (\ref{cuntzequal}) that $f$ and $g$ are everywhere nonzero, whence so is the homotopy $(1-t)f +tg$ for every $t \in [0,1]$.  This proves the theorem in the case that $A$ is purely infinite.

It remains to consider the situation where $A$ is stably finite and $f, g \in (C(X,A))_+$ satisfy (\ref{cuntzequal}) for some $a \in A_+$ with $\langle a \rangle$ non-compact. Let $(p_{i,f})_{i \in \mathbb{N}}$ and $(p_{i,g})_{i \in \mathbb{N}}$ be the sequences of projections provided by Proposition \ref{projsum} for $f$ and $g$, respectively.  In order to access the machinery of the proof of Theorem 1.1 of \cite{To2} we relabel the sequences as follows:
\[
p_i := p_{i.f}, \ q_i:= p_{i,g}.
\]
Appealing to Proposition \ref{projsum} we may now simply assume without loss of generality that
\begin{equation}\label{fg}
f = \bigoplus_{i=1}^\infty \frac{1}{2^i} p_i \ \mathrm{and} \ g = \bigoplus_{i=1}^\infty \frac{1}{2^i} q_i,
\end{equation}
where the sequences $(p_i)_{i \in \mathbb{N}}$ and $(q_i)_{i \in \mathbb{N}}$ satisfy:
\begin{enumerate}
\item[(i)] $p_ip_j = 0 = q_iq_j$ for every pair $i \neq j$;
\item[(ii)] $p_i \sim q_i$ for each $i \in \mathbb{N}$;
\item[(iii)] $p_i,q_i \neq 0, 1_A$ for each $i \in \mathbb{N}$;
\item[(iv)] $p_i$ and $q_i$ are full for each $i \in \mathbb{N}$.
\end{enumerate}
From this point the construction of the desired homotopy between $f$ and $g$ follows the construction of the homotopy between the elements $a$ and $b$ in the proof of Theorem 1.1 of \cite{To2} verbatim---these elements $a$ and $b$ have the form of $f$ and $g$, respectively, as in (\ref{fg}), and satisfy (i)--(iv) above.  This establishes Theorem \ref{main}.
\end{proof}

\end{document}